\documentclass[a4paper,12pt]{article}
\usepackage[english]{babel}
\usepackage[latin1]{inputenc}
\usepackage[T1]{fontenc}
\usepackage{sectsty}
\usepackage{amsthm}
\usepackage{makeidx}
\usepackage{latexsym,amsmath,amsfonts,amscd,amssymb}
\usepackage{graphicx}
\usepackage{xypic}
\usepackage{array,booktabs,arydshln}
\usepackage{tikz}
\usepackage{tabularx}
\usepackage{setspace}
\usepackage{fancyhdr}
\usepackage{fancybox}
\usepackage[all,cmtip]{xy}
\usepackage{pdfpages}
\usepackage{enumitem}
\usepackage{authblk}

\sectionfont{\centering}
\subsectionfont{\centering}

\newtheorem{theorem}{Theorem}[section]
\newtheorem{mthm}{Main Theorem}
\newtheorem{prop}[theorem]{Proposition}

\newtheorem{lemma}[theorem]{Lemma}

\theoremstyle{definition}
\newtheorem{rmk}[theorem]{Remark}
\newtheorem{example}[theorem]{Example}
\newtheorem{definition}[theorem]{{\bf Definition}}
\newtheorem{mthmX}{Main Theorem}
\newtheorem*{definition*}{{\bf Definition}}

\newcommand{\opn}{{\mathcal{O}_{\mathbb{P}^3}}}
\newcommand{\pn}{{\mathbb{P}^3}}

\newcommand{\HH}{{\mathcal H}}

\newcommand{\FF}{\mathcal{F}}
\newcommand{\TT}{\mathcal{T}}

\newcommand{\Ag}{\mathcal{A}^{\gamma(t)}}
\newcommand{\BB}{\mathcal{B}^{\beta(t)}}

\newcommand{\im}{\operatorname{Im}}
\newcommand{\codim}{\operatorname{codim}}
\newcommand{\coker}{\operatorname{coker}}

\newcommand{\ch}{\operatorname{ch}}
\newcommand{\Coh}{\operatorname*{Coh}(\mathbb{P}^3)}
\newcommand{\Db}{\operatorname*{D}^b(\mathbb{P}^3)}

\newcommand{\gsk}{\operatorname{GS}_k}
\newcommand{\gs}{\operatorname{GS}}
\newcommand{\reg}{\operatorname{reg}}
\newcommand{\Supp}{\operatorname{Supp}}

\title{Zero Rank Asymptotic Bridgeland Stability}
\author{Victor do Valle Pretti}
\affil{IMECC - UNICAMP \\
Departamento de Matem\'atica \\ Rua S\'ergio Buarque de Holanda, 651 \\
13083-970 Campinas-SP, Brazil\\
v120238@dac.unicamp.br}
\date{}

\begin{document}
\maketitle

\begin{abstract}
In this paper we examine the conditions that an object $E$ with $\ch_0(E)=0$ has to satisfy in order for it to be asymptotically (semi)stable with regard to Weak or Bridgeland stability conditions. This notion turned out to be equivalent to sheaf Gieseker-Simpson (semi)stability or a dual of it, depending on the curve considered.
\end{abstract}

\section{Introduction}

 The concept of Bridgeland stability condition was introduced by Bridgeland in \cite{Bri1} as a way to define an invariant for triangulated categories through a mathematician's point of view of Douglas's work on $\Pi$-stability. In a later paper \cite{Bri2}, Bridgeland uses torsion pairs to parametrize a half-plane of stability conditions over a smooth K$3$ surface and introduces the study of the large volume limit to these stability conditions. These techniques are generalized in \cite{BMT} by defining the notion of tilt stability and also conjecturing the existence of Bridgeland stability conditions over smooth projective threefolds.
 
 The existence of Bridgeland stability conditions in threefolds is conditioned to the existence of a generalized Bogomolov-type inequality for tilt stability as proved in \cite{BMT}. This inequality is proved in several cases: for $\mathbb{P}^3$ by Macr\`{\i} \cite{Mac3}, for quadric threefold by Schimdt \cite{Sch1}, for Fano threefolds with Picard rank $1$ and the quintic threefold by Li \cite{Li1,Li2}, for all abelian threefolds by Bayer--Macr\`{\i}--Stellari \cite{BMS} and principally polarized abelian threefolds with Picard rank $1$ by Maciocia--Piyaratne \cite{AP,AP2}. There is no known proof, to this point, of a general existence theorem for the Bridgeland stability conditions over smooth projective threefolds.
 
In fact, it has been proved that there are varieties for which a generalized Bogomolov inequality does not hold. Moreover, it is known that the generalized Bogomolov inequality, as proposed in \cite{BMT}, is a sufficient condition to prove the existence of Bridgeland stability conditions but it is not a necessary condition, see \cite{MS}.
 
 In \cite{AMac}, Maciocia defines the notion of numerical and actual walls for Bridgeland stability conditions and proves that, over surfaces, they behave like semi-circles in the half-plane. This is then generalized by Schmidt to tilt stability conditions over threefolds in \cite{Sch1}. In the threefold case it is more difficult to describe the walls for Bridgeland stability conditions as they are quartic curves in the half-plane and may be unbounded.
 
 To study the structure of these walls, Jardim and Maciocia define both Bridgeland and tilt asymptotic stability with respect to an unbounded curve $\gamma$. They were able to prove that over a special class of curves the notion of asymptotic tilt (semi)stability is related to $\gs_{2}$-(semi)stability and that asymptotic Bridgeland (semi)stability is related to Gieseker-(semi)stability. We will use the following notion of asymptotic stability throughout the paper:
 
 \begin{definition*}
Let $\sigma_{\beta,\alpha}=(Z_{\beta,\alpha},C^{\beta,\alpha})$ be a family of weak stability condition parametrized by $(\beta,\alpha) \in \mathbb{H}$, $\phi_{\beta,\alpha}$ its slope and $\gamma$ an unbounded curve in $\mathbb{H}$.  For an object $A \in \Db$ to be asymptotic $\phi_\gamma$-(semi)stable it has to satisfy the following conditions:
\begin{enumerate}[label=(\alph*)]
    \item There exists $t_0 \in \mathbb{R}$ such that $A \in C^{\gamma(t)}$ for $t>t_0$,
    
    \item Suppose that there exists $t_1\geq t_0 \in \mathbb{R}$ such that $F \stackrel{f}{\hookrightarrow} A \in C^{\gamma(t)}$ for some $F \in \Db$ and every $t>t_1$, then exists $t_2 \in \mathbb{R}$ with $\phi_{\gamma(t)}(F) <(\leq) \phi_{\gamma(t)}(A)$ whenever $t>t_2$.
\end{enumerate}
\end{definition*}

 It is important to note that in \cite{JM} there is a different version of asymptotic stability, a stronger variation. This stronger version is equivalent to saying that whenever we choose an unbounded curve $\gamma$, an object $A$ is strongly asymptotic $\lambda_\gamma$-(semi)stable if exists a last actual wall intersecting $\gamma$ and $A$ is asymptotic $\lambda_\gamma$-(semi)stable, that is, if there exists $t_0$ such that $A$ is $\lambda_{\gamma(t),s}$-(semi)stable for all $t>t_0$. The paper proves the existence of a last actual wall for all sheaves in an specific curve and what are the conditions sufficient and necessaries for a sheaf $E$ with Chern character $\ch_0(E)\neq 0$ to be asymptotic $\phi_\gamma$-(semi)stable.

 In this paper we try to fill the gap left in the paper, proving the sufficient and necessary condition for a sheaf $E$ with Chern character $\ch_0(E)=0$ to be $\phi_\gamma$-(semi)stable, for an unbounded curve $\gamma(t)=(\beta(t),\alpha(t))$ in $\mathbb{H}$. We have the following results:
 
 \begin{mthmX}
Let $\underset{t \rightarrow +\infty}{\lim}\beta(t)=-\infty$ and $c_\gamma<1$. Suppose that $E \in \Db$ with $\ch_0(E)=0$, for $E$ to be asymptotic $\lambda_{\gamma,s}$-(semi)stable it is necessary and sufficient that $E \in \Coh$ is a Gieseker-(semi)stable sheaf.
\end{mthmX}

 \begin{mthmX}
Let $\underset{t \rightarrow +\infty}{\lim}\beta(t)=+\infty$ and $c_\gamma<1$. An object $E \in \Db$ with $\ch_0(E)=0$ is asymptotic $\lambda_{\gamma}$-(semi)stable if and only if it is the dual of a Gieseker-(semi)stable sheaf. 
\end{mthmX}

 \textbf{Acknowledgments:} I'd like to thank both my advisors Prof. Marcos Jardim and Prof. Antony Maciocia for the opportunity to work on this topic and the incredible support in the development of this research. This work is supported by FAPESP research grant 2016/25249-0 and 2019/05207-9. Also, I'd like to thank FAPESP for funding a number of grants throughout my education and making this research possible. This project was partially developed at the University of Edinburgh to which I am grateful for the hospitality and support.

\section{Background}

This section is where we define most of the notations and concepts we are going to study throughout the paper. It starts with a broad definition of weak and Bridgeland stability conditions, with examples constructed using tilts of hearts of bounded $t$-structures.

Next we begin to define some curves called \emph{distinguished curves} which will be important when trying to understand whether objects lie in $\mathcal{A}^{\beta,\alpha}$ or $\mathcal{B}^{\beta}$, later this will be used to describe what are the objects at infinity. This subsection finishes with an important definition of unbounded curves we will be interested in.

The final section is dedicated to defining our main concept, namely \emph{asymptotic stability}, and its sheaf counter-part $\gsk$-stability. We also prove here the Gieseker-stability of the structure sheaf of smooth subvarieties and a restatement in our notation of a classic Grothendieck Theorem about boundedness of families of sheaves.

\subsection{Stability conditions}

 Let $X$ be a smooth projective irreducible variety of dimension $3$ with Picard rank $1$ over $\mathbb{C}$, $K(X)$ its Grothendieck group, $\Gamma$ a finite rank lattice and $v : K(X) \twoheadrightarrow \Gamma$ a surjective group homomorphism. We will start by defining the stability conditions we will work with. These are given by a heart of a bounded t-strucure in $D^b(X)$ and a homomorphism from the Grothendieck group $K(X)$ to the complex numbers, it will be a generalization of the $\mu$-slope stability for sheaves. 

\begin{definition}
A weak stability condition is a pair $\sigma=(Z,\mathcal{A})$, where $Z: K(X) \rightarrow \mathbb{C}$ is a group homomorphism and $\mathcal{A}$ is a heart of a bounded $t$-structure satisfying:

\begin{itemize}
    \item(Positivity) For every $E \in \mathcal{A}$ we have $ Z(E)\in \mathbb{R}_{\geq 0} \cdot e^{i\pi \phi}$ for some $\phi \in [0,1)$. The phase of $E$ is defined as $\phi_\sigma(E):=-\Re(Z(E))/\Im(Z(E))$ and $\phi_\sigma(E)=+\infty$ if $Z(E)=0$.
    
A object $A \in \mathcal{A}$ is called $\phi_\sigma$-(semi)stable if $\phi(A)>(\geq)\phi(F)$ for every non-zero subobject $F$ of $A$ in $\mathcal{A}$.
    
    \item(Harder--Narasimhan filtration) Let $E \in \mathcal{A}$, then there exists $n \in \mathbb{Z}_{>0}$ and $E_0,...,E_n \in \mathcal{A}$ such that
    
    \begin{center}
  $E_0=0 \subset E_1 \subset E_2 \subset ... \subset E_{n-1} \subset E=E_n$
  \end{center}
  
  with $F_i= E_i/E_{i-1}$ $\phi_\sigma$-semistable objects such that $\phi_\sigma(F_i)>\phi_\sigma(F_{i-1})$ for all $1\leq i \leq n$.
  
  \item(Numerical condition) There is a factorization for the homomorphism $Z$ such as \\
 \centerline{
\xymatrix{
K(X) \ar[rr]^{Z} \ar[rd]^{v} &&\mathbb{C}\\
&\Gamma_\mathbb{Q} \ar[ru]^{\tilde{Z}}
}}

\end{itemize}
\end{definition}

To define a Bridgeland stability we have to impose a few extra conditions on \linebreak $\sigma=(Z,\mathcal{A})$, namely the non-existence of objects in the kernel of the stability function $Z$ and the support property. The support property is usually the most difficult part to be proven when trying to construct a new Bridgeland stability condition, as it is related to the existence of a generalized Bogomolov-type inequality.

\begin{definition}\label{def-Bri}
A weak stability condition $\sigma=(Z,\mathcal{A})$ is a Bridgeland stability condition if satisfies:
\begin{enumerate}[label=(\alph*)]
    \item There is no object $E$ in $\mathcal{A}$ such that $Z(E)=0$. 
    
    \item(Support property) Let $|| \cdot ||_{\mathbb{R}}$ be a norm in $\Gamma_\mathbb{R}$. Then
 
 \begin{center}
 $C_\sigma:= \inf \Big\{ \frac{|Z(E)|}{||v(E)||} : \text{$E \neq 0$ and $\phi$-semistable}\Big\} >0$. 
 \end{center}
    
\end{enumerate}
\end{definition}

Every weak stability condition $\sigma=(\mathcal{A},Z)$ satisfy a weak seesaw property, that is, for every exact sequence in $\mathcal{A}$
\[ 0 \rightarrow E \rightarrow F \rightarrow G \rightarrow 0,\]
either all $\phi_\sigma$-slopes satisfy $\phi_{\sigma}(E)\leq \phi_{\sigma}(F) \leq \phi_{\sigma}(G)$ or $\phi_{\sigma}(E) \geq \phi_{\sigma}(F) \geq \phi_{\sigma}(G)$. In a Bridgeland stability condition we have a more refined seesaw property saying that the phase of objects in an exact sequence either all strictly increase or all strictly decrease or are all equal.

We will be working with $X=\pn$ as it is one notable example where Bridgeland Stability Conditions are known to exist, see \cite{Mac3}. Fix $\Gamma= \mathbb{Z} \oplus \mathbb{Z} \oplus \frac{1}{2} \mathbb{Z} \oplus \frac{1}{6} \mathbb{Z}$ as the finite rank lattice, H as the hyperplane cycle in $K_{num}(X)=\Gamma$ and $v: K(X) \twoheadrightarrow \Gamma$ as
\[ v(E) = (\ch_0(E)\cdot H^3 , \ch_1(E)\cdot H^2 , \ch_2(E)\cdot H , \ch_3(E)), \]
where $\ch_i(E) \in A^i(X)\otimes \mathbb{Q}$ and $\cdot$ represents the intersection of these classes in the rational Chow ring $A(X)\otimes\mathbb{Q}$. From now on set the notation as $\ch_i(E)=v_i(E)$. Let $\beta$ be a real number and define 
\[ \ch_0^\beta(E)=\ch_0(E)\]
\[\ch_1^\beta(E) = \ch_1(E) - \beta \ch_0(E)\]
\[\ch_2^\beta(E) = \ch_2(E) - \beta \ch_1(E) + \frac{1}{2} \beta^2 \ch_0(E)\]
\[\ch_3^\beta(E) = \ch_3(E) - \beta \ch_2(E) + \frac{1}{2} \beta^2 \ch_1(E) - \frac{1}{6} \beta^3 \ch_0(E).\]
\begin{rmk}
Throughout the paper we will define functions using either $E\in \Db$ or $v \in \Gamma_\mathbb{Q}=\Gamma\otimes\mathbb{Q}$, this definitions will be interchangeable by making $v=\ch(E)$.
\end{rmk}

In \cite{BMT} Bayer--Macr\`i--Toda conjectured a way to generate Bridgeland stability conditions for threefolds, which later was proven to hold for a few cases of threefolds, see \cite{BMS,BMSZ,Li1,Li2}. The idea was a generalization to the threefold case of the approach given in \cite{Bri1} by using the stability notion for surfaces as a weaker notion of stability for threefolds and apply a new tilting to the already tilted category. 

Start by considering
\[\mathcal{T}_\beta :=\{ E \in Coh(X) |\text{ for every $E \twoheadrightarrow Q$, $\mu(Q) > \beta$}\},\]
\[\mathcal{F}_\beta :=\{ E \in Coh(X) | \text{ for every $F \hookrightarrow E$, $\mu(F)\leq \beta$}\} \] 
subsets of $\Coh$ such that $(\mathcal{T}_\beta,\mathcal{F}_\beta)$ form a torsion pair in $\Coh$, as in \linebreak \cite[Definition 3.2]{Bri2}, for every $\beta \in \mathbb{R}$. Note that all sheaves in $\mathcal{F}_\beta$ are torsion-free. Applying a tilt in $\Coh$ we obtain a new heart of bounded $t$-structure and denote it by $\mathcal{B}^\beta$. This family of categories can be used to create a family of weak stability conditions known as \emph{tilt stability conditions}. Let $\alpha \in \mathbb{R}_{>0}$, $E \in \Db$ and
\begin{equation}\label{tilt}
    Z_{\beta,\alpha}^{t}(E) = \left(- \ch_2^\beta(E) + \frac{1}{2}\alpha^2 \ch_0(E)\right) + i \ch_1^\beta(E),
\end{equation}
such that $\sigma^t_{\beta,\alpha}=(Z_{\beta,\alpha}^t,\mathcal{B}^\beta)$ is the desired weak stability condition, as proved in \cite{Sch1}. Denote by $\nu_{\beta,\alpha}$ the phase of $\sigma^t_{\beta,\alpha}$.

This stability condition was extensively studied in the surface case, see \cite{AMac,BD}, and this knowledge was later applied to the threefold case in \cite{Sch1}. To use the construction done in \cite{BMT}, define a new torsion pair, now in $\mathcal{B}^\beta$, as
\[\mathcal{T}_{\beta,\alpha} := \{E \in \mathcal{B}^\beta | \text{ for every $E \twoheadrightarrow Q$, $\nu_{\beta,\alpha}(Q)>0$}\},\]
\[\mathcal{F}_{\beta,\alpha} := \{E \in \mathcal{B}^\beta | \text{ for every $F \hookrightarrow E$, $\nu_{\beta,\alpha}(F)\leq 0$}\}\] 
and applying a tilt to $\mathcal{B}^\beta$ in order to obtain the new family of hearts of bounded $t$-structures denoted by $\mathcal{A}^{\beta,\alpha}$. The group homomorphism we will be using as a stability function is defined as follows
\[Z_{\beta,\alpha,s}(E) = \left(-\ch_3^\beta(E) + \left(s+\frac{1}{6}\right)\ch_1^\beta(E)\right)+ i \left(\ch_2^\beta(E)-\frac{1}{2}\alpha^2\ch_0(E)\right),\]
for $E \in \Db$, so that $\sigma_{\beta,\alpha,s}=(Z_{\beta,\alpha,s},\mathcal{A}^{\beta,\alpha})$ is a Bridgeland stability condition, for every $s \in \mathbb{R}_{>0}$. In \cite{BMT}, they proved that $\sigma_{\beta,\alpha,s}$ is a weak stability condition satisfying condition $(a)$ in Definition \ref{def-Bri}, the support property was conjectured to be true and associated to the existence of a generalized Bogomolov-type inequality. 

\begin{definition}
Let $E$ be $\nu_{\beta,\alpha}$-semistable object in $\mathcal{B}^{\beta}$. Then $\ch(E)$ satisfy the Generalized Bogomolov inequality:
\[ Q_{\beta,\alpha}(E):=(\ch_1(E)^2-2\cdot \ch_2(E)\ch_0(E))\alpha^2+4(\ch_2^\beta(E))^2-6\ch_1^\beta(E)\ch_3^\beta(E)\geq 0.\]
\end{definition}

Where \[Q^{\text{tilt}}(E)=\ch_1^2(E)-2\cdot \ch_2(E)\ch_1(E)\] is a quadratic form such that $Q^{\text{tilt}}(E)\geq 0$ for any $\mu$-semistable sheaf E, a classical result known as the Bogomolov inequality. 

\begin{rmk}
One of the consequences of $(Z_{\beta,\alpha,s},\mathcal{A}^{\beta,\alpha})$ being a Bridgeland stability condition is that for all $E \in \mathcal{A}^{\beta,\alpha}$, $\Im(Z_{\beta,\alpha,s}(E))\geq 0$. 
\end{rmk}

Due to its nature, being generated by a double tilt, $\mathcal{A}^{\beta,\alpha}$ has objects with $\mathcal{H}^{-i}$ $\Coh$-cohomology equal to $0$ for $i \neq 0,1,2$ and we also have a $\mathcal{B}^\beta$-cohomology functor $\mathcal{H}^i_\beta: \Db \rightarrow \mathcal{B}^\beta$ such that these cohomology functors are related by the following distinguished triangles described in \cite{JM} for every object $E \in \mathcal{A}^{\beta,\alpha}$:
\begin{equation}\label{Exa-1}
 0 \rightarrow \mathcal{H}^{-1}_\beta(E)[1] \rightarrow E \rightarrow \mathcal{H}^0_\beta(E) \rightarrow 0
\end{equation}
\begin{equation}\label{Exa-2}
 0 \rightarrow \mathcal{H}^{0}(\mathcal{H}^{-1}_\beta(E)) \rightarrow \HH^{-1}(E) \rightarrow \mathcal{H}^{-1}(\mathcal{H}^0_\beta(E)) \rightarrow 0
\end{equation}
with $\mathcal{H}^{-2}(E)=\mathcal{H}^{-1}(\mathcal{H}^{-1}_\beta(E))$, $\mathcal{H}^0(E)=\mathcal{H}^0(\mathcal{H}^0_\beta(E))$, sequence \eqref{Exa-1} being exact in $\mathcal{A}^{\beta,\alpha}$ and sequence \eqref{Exa-2} being exact in $\Coh$.

\begin{rmk}
Let $E$ be an object in $\mathcal{A}^{\beta,\alpha}$ and also in $\mathcal{A}^{\beta',\alpha'}$, where $\beta \neq \beta'$ then $\mathcal{H}^i (E)$ does not vary with $\beta$ but $\mathcal{H}^{i}_\beta(E)$ is not necessarily equal to $\mathcal{H}^{i}_{\beta'}(E)$, thus $\ch_k(\mathcal{H}^{i}_{\beta}(E))$ may vary whenever we change $\beta$. One of the technical difficulties we will have to address is to prove that $\underset{t\rightarrow +\infty}{\lim}\ch_i(\mathcal{H}^{i}_{\beta(t)}(E))$ exists, whenever $\underset{t \rightarrow +\infty}{\lim}\beta(t)=\pm \infty$ . 

To exemplify this consider the sheaf $\opn$. Since $\opn$ is a $\mu$-stable sheaf we can see that $\opn \in \mathcal{B}^{\beta}$ when $\beta<0$ and $\opn[1] \in \mathcal{B}^{\beta}$, otherwise. Therefore, for negative $\beta$ we have $\mathcal{H}^{-1}_{\beta}(\opn[1])=\opn$ and $\mathcal{H}^{0}_\beta(\opn[1])=0$, conversely for positive $\beta$ we have $\mathcal{H}^{-1}_\beta(\opn[1])=0$ and $\mathcal{H}^{0}_\beta(\opn[1])=\opn[1]$. It is important to note that if $-\beta<\alpha$ and $\alpha>\beta$ then $\opn[1] \in \mathcal{A}^{\beta,\alpha}$, for either case of $\beta$ positive or negative.

\end{rmk}

The nature of $\mathcal{A}^{\beta,\alpha}$ implies certain conditions on its objects, one of those is expressed by the following lemma.

\begin{lemma}[\cite{JM}]\label{L-ref}
If $A$ is an object in $\mathcal{A}^{\beta,\alpha}$, then $\mathcal{H}^{-2}(A)$ is a reflexive sheaf of dimension $3$. 
\end{lemma}

\subsection{Distinguished Curves}\label{Sec-Cur}

In Bridgeland stability we have a decomposition of the space of stability conditions in walls and chambers, with respect to a fixed Chern character $v\in K(X)$, see \cite[Section 9]{Bri2}. This decomposition is important because the moduli of Bridgeland stable objects with fixed Chern character $v$ does not vary whenever we change the stability inside a chamber, only when crossing walls is that the moduli can change. In this section we define the tools necessary to study the geometry of the upper half-plane sections of walls inside the space of Bridgeland stability with respect to a fixed Chern character.

\begin{definition}
A \emph{numerical wall} inside the space of (weak)Bridgeland stability conditions with respect to an element $w \in \Gamma$ is the subset of  stability conditions $\sigma=(Z,\mathcal{A})$ with non trivial solutions to the equation $\phi_\sigma(v)=\phi_\sigma(u)$ for some $u \in \Gamma$.

A subset of a numerical wall is called an \emph{actual wall} if, for each point $\sigma=(\mathcal{A},Z)$ in this subset, there is a sequence of $\phi_\sigma$-semistable objects $0 \rightarrow A \rightarrow B \rightarrow C \rightarrow 0$ in $\mathcal{A}$ with $v(B)=w$ and $\phi_\sigma(A)=\phi_\sigma(B)=\phi_\sigma(C)$.  
\end{definition}

The previous definition is for the whole space of stability conditions,  we will focus on the subset of these walls cut by $\mathbb{H}$, that is, points on numerical walls that come from $\sigma_{\beta,\alpha,s}$($\sigma^t_{\beta,\alpha}$). These numerical and actual walls can be tricky to work in the case of Bridgeland stability conditions $\sigma_{\beta,\alpha,s}$ as they are plane curves in degree $4$ but they have nice behaviour along some notable curves as proven in \cite{JM} and \cite{Sch1}.

\begin{definition}
Let $w,v \in \Gamma$ and define the following curves:
\begin{itemize}
    \item  The numerical wall for $v$ and $w$ in $\mathbb{H}$ for $\sigma^t_{\beta,\alpha}$ is denoted by $\Sigma_{v,w}$, known as the $\nu$-wall associated with $v$ and $w$.
    \item  Let $s$ be a positive real number, the numerical wall for $v$ and $w$ in $\mathbb{H}$ for $\sigma_{\beta,\alpha,s}$ is denoted by $\Upsilon_{v,w,s}$, known as the $\lambda$-wall associated with $v$ and $w$.
    
    \item  $L_w:=\{(\beta,\alpha) \in \mathbb{H}| \ch_1(w)=\beta\ch_0(w)\}$. The space to the left of the line $L_w$ will be denoted by $L_w^+:=\{\ch_1^\beta(w)> 0\}$ and respectively the right-hand side will be $L_w^-:=\{\ch_1^\beta(w)<0\}$.
    \item  $\Theta_w:=\{(\beta,\alpha) \in \mathbb{H}| \Re(Z_{\beta,\alpha}^t(w))=0\}$. Theta may divide the plane in two regions $\Theta_w^+:=\{\Re(Z_{\beta,\alpha}^t(w))>0\}$ and $\Theta_w^-:=\{\Re(Z_{\beta,\alpha}^t(W))<0\}$.
    \item  $\Gamma_w:=\{(\beta,\alpha) \in \mathbb{H}| \Re(Z_{\beta,\alpha,s}(w))=0\}$.
\end{itemize}
\end{definition}

\begin{example}\label{Ex-walls}
Let $E$ be a $\mu$-semistable sheaf and consider what is necessary for $(\beta,\alpha) \in \mathbb{H}$ to imply either that $E \in \mathcal{B}^{\beta}$ or $E \in \mathcal{A}^{\beta,\alpha}$. The first observation is that $E \in \mathcal{B}^{\beta}$ whenever $\beta<\mu(E)$, to the left of $L_E$ in $\mathbb{H}$, and $E \in \mathcal{B}^{\beta}[-1]$ otherwise. If $E$ is $\nu_{\beta,\alpha}$-semistable then there are $3$ possiblities:
\begin{itemize}
    \item $E \in \mathcal{A}^{\beta,\alpha}$ if $(\beta,\alpha) \in L_w^+ \cap \Theta_w^+ $
    \item $E \in \mathcal{A}^{\beta,\alpha}[-1]$ if $(\beta,\alpha) \in (L_w^- \cap \Theta_w^-)\cup (L_w^+ \cap \Theta_w^-)$ or 
    \item $E \in \mathcal{A}^{\beta,\alpha}[-2]$ if $(\beta,\alpha) \in L_w^-\cap \Theta_w^+$.
\end{itemize}
\end{example}

\begin{figure}[htp]
    \centering
    \includegraphics[width=10cm]{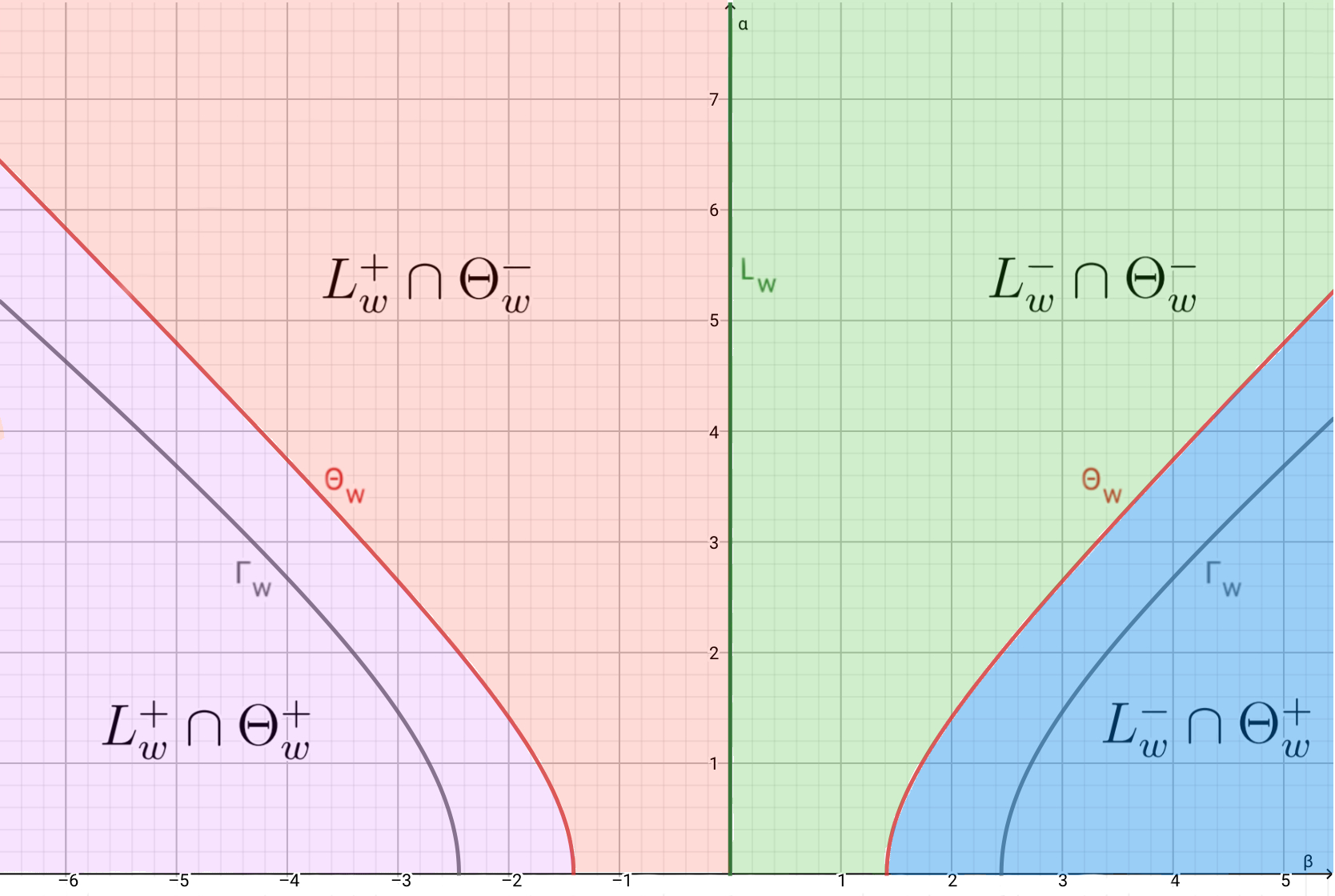}
    \caption{Distinguished curves related to $w=(2,0,-2,0)=\ch(I)$, where $I$ is an Instanton sheaf of charge $2$.}
\end{figure}

 In \cite{Sch1}, Schmidt uses the ideas in \cite{AMac} to prove a structure theorem for walls in tilt stability in threefolds. We restate it in a slightly different form that suits our purpose. 

\begin{theorem}\label{Struc-Til}(\cite{Sch1,AMac})
Fix a vector $w=(R,C,D,E) \in \Gamma$. The walls are with respect to $w$.

\begin{enumerate}[label=(\alph*)]
\item Numerical $\nu$-walls are of the form
\[
x \alpha^2 +x \beta^2 + y \beta +z = 0
\]
for $x= Rc - Cr$, $y= 2(Dr - Rd)$ and $z=2(Cd - Dc)$. In particular, they are all semicircles with center at the $\beta$-axis or vertical rays.
\item Two numerical $\nu$-walls intersect if and only if they are identical.
\item If $R \neq 0$, the curve $\Theta_w$ is given by the hyperbola  
\[
(\beta - C/R)^2 - \alpha^2 = \frac{Q^{\text{tilt}}(w)}{R^2}
\]
\item If a numerical $\nu$-wall is an actual $\nu$-wall for some point then it is an actual $\nu$-wall at every point.
\item If $\Sigma_{w,v}$ is a numerical $\nu$-wall then $\Sigma_{w,v}\cap\Theta_w=\Sigma_{w,v}\cap \Theta_v=\Theta_w \cap \Theta_v=\{p\}$ is the only point in the semi-circle $\Sigma_{w,v}$ with horizontal tangent space. 
\end{enumerate}
\end{theorem}

Theorem \ref{Struc-Til} takes care of the structure and geometry of the $\nu$-walls, at least in a general setting. To understand how numerical $\lambda$-walls behave we will direct to \cite{JM}, more specifically Sections $4$ and $6$. We will only need the equation that determines the numerical $\lambda$-wall for any $v,w \in \Gamma$ as in \cite{JM} and for that let \[\delta_{ij}(v,w)=\ch_i(v)\ch_j(w)-\ch_j(v)\ch_i(w).\] When the objects are known from the context we will denote $\delta_{ij}(v,w)$ by $\delta_{ij}$. 

A point $(\beta,\alpha) \in \mathbb{H}$ is in $\Upsilon_{v,w,s}$, for some $v,w \in \Gamma$ and $s>0$, if it satisfies
\[
f_{v,w}(\beta,\alpha,s)=\frac{6s+1}{12}\delta_{10}\alpha^4+\]\[\left(\frac{3s-1}{6}\delta_{10}\beta^2+\frac{1-3s}{3}\delta_{20}\beta+\frac{6s+1}{6}\delta_{21}-\frac{1}{2}\delta_{30} \right)\alpha^2+\]\[\left(\frac{1}{12}\delta_{10}\beta^4-\frac{1}{3}\delta_{20}\beta^3+ \frac{\delta_{30}+\delta_{21}}{2}\beta^2-\delta_{31}\beta+\delta_{32} \right).
\]

Throughout the rest of the paper we will fix the notation that $\gamma$ is an unbounded curve in $\mathbb{H}$ parametrized by $\gamma(t):=(\beta(t),\alpha(t))$ satisfying \begin{equation}\label{Eq-Unbou}
c_{\gamma}:=\underset{t \rightarrow +\infty}{\lim}{\frac{\alpha^2(t)}{\beta^2(t)}}<1.  \end{equation}

This condition is important because it is the one that guarantees that whenever $E \in \Coh$ we have that $\underset{t \rightarrow +\infty}{\lim}\nu_{\gamma(t)}(E)>0$ (see equation \eqref{Lim-nu}), in geometric terms it means that for every $E \in \Coh$ there is some $t_0\in \mathbb{R}$ such that when $t>t_0$, $\gamma(t) \in \Theta_E^+$ holds. In $\cite{JM}$ these curves are called $\Theta^-$-curves.

\subsection{Asymptotic stability}

As seen in the previous subsection, the geometry and structure of the walls for tilt stability in threefolds is described by Lemma \ref{Struc-Til}, but not much is known about the behaviour of these walls for Bridgeland stability conditions. This is due to their definition, being done so by zeros of a degree $4$ polynomial, and also that $\mathcal{A}^{\beta,\alpha}$ is much more complicated than $\mathcal{B}^{\beta}$, being a tilt of $\mathcal{B}^{\beta}$ it may have $3$-step object which do not occur in $\mathcal{B}^\beta$. 

One way to study the stability of objects in $\Db$, circumventing this difficulty, is proposed in \cite{JM} by trying to understand how objects behave asymptotically at infinity. It turned out to be closely related to sheaf Gieseker-stability. 

\begin{definition}\label{Def-Asy}
Let $\sigma_{\beta,\alpha}=(Z_{\beta,\alpha},C^{\beta,\alpha})$ be a famility of weak stability condition parametrized by $(\beta,\alpha) \in \mathbb{H}$ and $\phi_{\beta,\alpha}$ its slope.  For an object $A \in \Db$ to be asymptotic $\phi_\gamma$-(semi)stable it has to satisfy the following conditions:
\begin{enumerate}[label=(\alph*)]
    \item There exists $t_0 \in \mathbb{R}$ such that $A \in C^{\gamma(t)}$ for $t>t_0$,
    
    \item Suppose that there exists $t_1\geq t_0 \in \mathbb{R}$ such that $F \stackrel{f}{\hookrightarrow} A \in C^{\gamma(t)}$ for some $F \in \Db$ and every $t>t_1$, then exists $t_2 \in \mathbb{R}$ with $\phi_{\gamma(t)}(F) <(\leq) \phi_{\gamma(t)}(A)$ whenever $t>t_2$.
\end{enumerate}
\end{definition}

The cases we will be applying this definition are when $\sigma_{\beta,\alpha}$ are either tilt or Bridgeland stability conditions obtained by the tilt algorithm we described in the previous subsection, and $\phi_{\beta,\alpha}=\nu_{\beta,\alpha}$ or $\phi_{\beta,\alpha}= \lambda_{\beta,\alpha,s}$, for a fixed $s$, respectively. 

It is important to note that this definition is equivalent to using quotients instead of subobjects in item (b), because of the way exact sequences are defined in hearts of $t$-structures. Using the notation of Definition \ref{Def-Asy}, a sequence in $C^{\beta,\alpha}$ is exact if the corresponding triangle in $\Db$ is distinguished. Therefore, since we are fixing the map $f$, we know that $f$ is monomorphic in $C^{\beta,\alpha}$ if and only if its quotient in $C^{\beta,\alpha}$ is the cone $C(f) \in \Db$.

We finish the section by establishing the notion of Gieseker-Simpson (semi)stability and its variations as defined in \cite{SS} and \cite{JM}.

We will be using the nomenclature found in $\cite{HL}$ : let $E\in \Coh$ with dimension $d$, $E$ is called a pure sheaf if $\dim(F)=d$ for all subsheaves $F$ of $E$, similarly a torsion subsheaf of $E$ is any subsheaf $F$ with $\dim(F)<d$. In this situation, consider \linebreak $E^D:=\mathcal{E}xt^{3-d}(E,\opn)$ such that $E$ is said to be reflexive if $E\simeq E^{DD}$.

\begin{definition}
Let $E$ be a coherent sheaf over $X$ of dimension $d$ and $k$ be an integer $1 \leq k \leq d$. Let $P_E(n) = \chi( E \otimes \opn(n))$ be the Hilbert polynomial associated to $E$ with leading coefficient $\alpha_d$ then define

\[
p_{E,k}(t)= \sum_{i=d-k}^{d} (\alpha_i/\alpha_{d}) t^i
\]

We say that $E$ is $\gsk$-(semi)stable if $E$ is a pure sheaf and for every non zero subsheaf $A \hookrightarrow E$ we have $p_{E,k}(m) <(\leq) p_{E/A, k}(m)$ for $m\gg0$. If $k=\dim(E)$, a $\gsk$-(semi)stable $E$ is called Gieseker-(semi)stable and denote $p_{E,d}$ by $p_{E}$.
\end{definition}

It was observed by Schmidt that $\gsk$-stability implies $\gs_{k-1}$-stability and also $\gsk$-semistability implies $\gs_{k-1}$-semistability, for all $k\leq \dim(E)$. There is another, more useful in our situation, way of seeing $\gsk$-(semi)stability as a inequality involving the $\delta_{i,j}$ of the respective sheaves.

\begin{rmk}\label{Gie-delta}
We can use a factorization of $p_E(t)$ to establish an equivalent definition for $\gsk$-(semi)stability. By Grothendieck-Riemann-Roch we establish that for a $2$-dimensional pure sheaf $E$ and $F$ a subsheaf of $E$ we have
\[p_E(t)-p_F(t)= \frac{1}{\ch_1(F) \ch_1(E)}(\delta_{21}(E,F) x_1(t) + \delta_{31}(E,F)),\]
for some $x_1(t)$ linear polynomial. This implies that a $2$-dimensional pure sheaf $E$ is Gieseker-(semi)stable sheaf if and only if $(\delta_{21}(E,F),\delta_{31}(E,F))>(\geq) (0,0)$ in the lexicographical order. Similarly, $E$ is $\gs_1$-(semi)stable if and only if $\delta_{21}(E,F)>(\geq) 0$.

\end{rmk}

One example we will use throughout the paper is the structure sheaf of smooth subvarieties $S$ of $\mathbb{P}^3$. Let $i: S \hookrightarrow \mathbb{P}^3$ a smooth irreducible closed subscheme and $i_\ast\mathcal{O}_S$ the image of its structure sheaf in $\Coh$. We will show that $i_\ast\mathcal{O}_S$ is always Gieseker-stable. Before we start proving this, we will need to do a few observations about $i_\ast$ and $i^\ast$. Through the next remark define $E^\vee=R\mathcal{H}om(E,\opn)[2]$ for any $E\in \Db$.

\begin{rmk}\label{She-stru}
Since $S$ is a closed subscheme, then $i_\ast$ is an exact functor satisfying \linebreak $i^\ast(i_\ast(E))=E$ for all sheaves $E$ in $\text{Coh}(S)$. The functor $i^\ast$ is also exact as a consequence of $S$ being a smooth subscheme and, applying \cite[Section 3.4]{Huy}, we can define \[i^!(E):=i^\ast(E) \otimes \omega_i,\] where $\omega_i=\omega_S \otimes i^\ast \omega_{\mathbb{P}^3}[-c]$, with $c=\codim(S)$, $E \in \Db$, $\omega_S$ and $\omega_{\mathbb{P}^3}$ are the dualizing bundles of $S$ and $\mathbb{P}^3$, respectively. The last calculation we will need is\\ 
\hspace*{91pt}$i_\ast(i^\ast(E))=  i_\ast(i^!(E) \otimes \omega_i^\vee[-2])$ \hfill $(\omega_i \otimes \omega_i^\vee=\mathcal{O}_S[2])$\\
\hspace*{140pt}$ = i_\ast(i^!(E) \otimes i^\ast(i_\ast(\omega_i^\vee))[-2])$ \hfill $(i^\ast i_\ast=Id)$ \\
\hspace*{140pt}$ = i_\ast(i^!(E)) \otimes i_\ast(\omega_i^\vee)[-2]$  \hfill \text{(Projection Formula)}\\
\hspace*{140pt}$ = E \otimes i_\ast(\omega_i) \otimes i_\ast(\omega_i^\vee)[-2]$   \hfill (\text{Def. }$i^!$ \text{ and Proj. Formula}) \\
\hspace*{140pt}$ = E \otimes i_\ast \mathcal{O}_S$  \hfill$(i_\ast(\omega_i)\otimes i_\ast(\omega_i^\vee)=i_\ast\mathcal{O}_S[2])$

\end{rmk}

Let $F \overset{f}{\hookrightarrow} i_\ast\mathcal{O}_S$, $L=\coker(f)$ in $\Coh$ and consider the exact sequence, in $\text{Coh}(S)$, \[0 \rightarrow i^\ast(F) \rightarrow \mathcal{O}_S  \rightarrow i^\ast(L) \rightarrow 0.\] 

We have that $i^\ast(L)=\mathcal{O}_C$, for some subvariety $C$ of $S$, and since $S$ is irreducible we can see that $\dim(C)<\dim(S)$. Moreover, due to Remark \ref{She-stru}, this implies that $\ch_k(L)=0$, for $k=\codim(S)$, concluding our reasoning, because this implies that every subsheaf of $i_\ast\mathcal{O}_S$ satisfy the inequality in Remark \ref{Gie-delta}.

At last, let us remind a result proved in \cite[Lemma 1.7.9]{HL} and referenced as \emph{Grothendieck's Theorem}. It tell us about the boundedness of families of pure quotients of a given sheaf, and this will be important when proving the existence of the limits $\underset{t \rightarrow +\infty}{\lim}\ch_i(\mathcal{H}^j_{\beta(t)}(E))$. 

\begin{theorem}[Grothendieck Theorem]\label{Groth}
Let $E \in \Coh$ be a $d$-dimensional sheaf with $d>0$, Hilbert Polynomial $P$ and Mumford--Castelnuovo regularity $\reg(E)=p$. There exists $C$, depending only on $P$ and $p$, such that for every purely $d$-dimensional quotient $Q$ then $\ch_{4-d}^C(Q)\geq 0$. Moreover, the family of purely $d$-dimensional quotients $Q$ with $\ch_{4-d}^C(Q)$ bounded from above is a bounded family.
\end{theorem}

\section{Asymptotic $\nu_{\gamma}$-stability}

Due to its construction, $\sigma^t_{\beta,\alpha}=(Z_{\beta,\alpha}^t,\mathcal{B}^\beta)$ is much simpler than its Bridgeland counterpart, making it a great starting point to the study of asymptotic (semi)stable objects. For $\ch_0(E)=0$ objects we have to consider two cases: Either $\dim(\Supp(E))=2$ or $\dim(\Supp(E))\leq 1$, in the first case $\nu_\gamma$-stability will be equivalent to $\gs_1$-stability and on the latter case, every sheaf is $\nu_\gamma$-semistable. This is because $\nu_{\beta,\alpha}$ does not take into account $\ch_3$, making it a bad stability condition to distinguish low-dimensional sheaves.

    When we consider objects $E$ with $\ch_0(E)=0$ we are avoiding the existence of the canonical vertical wall $\{\beta=\mu(E)\}$, such wall is responsible for separating the regions $\{E \in \mathcal{B}^{\beta,\alpha}\}$ and $\{E \in \mathcal{B}^{\beta',\alpha'}[-1]\}$, if $E$ is a $\mu$-stable sheaf for example. This is the reason we are able to prove the same theorem for unbounded curves going either to the right or the left. We assume in this section that $\gamma$ is an unbounded curve satisfying $\underset{t \rightarrow +\infty}{\lim}{|\beta(t)|}=+\infty$.

\begin{prop}\label{Pro-nu}
Let $E \in \Db$ be an object with $\ch_0(E)=0$ and $\ch_1(E)\neq 0$. Then $E$ is asymptotic $\nu_\gamma$-(semi)stable if and only if it is $\gs_1$-(semi)stable.
\end{prop}

The version of this proposition related to the case $\ch_0(E)=\ch_1(E)=0$ is realized by knowing that $\Im(Z_{\beta,\alpha}^t(E))=\ch_1^\beta(E)=0$ for all $(\beta,\alpha) \in \mathbb{H}$, so that $E \in \mathcal{B}^{\beta,\alpha}$ for all $(\beta,\alpha)\in \mathbb{H}$ and $\nu_{\beta,\alpha}(F)=+\infty$ when $F \hookrightarrow E$ in $\BB$ 

\begin{proof}
To begin, assume that $E$ is asymptotic $\nu_\gamma$-(semi)stable. This implies a few properties about $E$: $E \in \BB$ for all $t$ greater than some $t_0$; $\HH^{i}(E)=0$ for all \linebreak $i\neq 0,-1$;  $\ch_0(\HH^{-1}(E))=\ch_0(\HH^0(E))$ and $\ch_1(\HH^{-1}(E))<\ch_1(\HH^0(E))$.
Let us prove that $\ch_0(\HH^{-1}(E))=\ch_0(\HH^0(E))=0$, if this was not the case then we would have that $\HH^{-1}(E)[1]$ and $\HH^0(E)$ are both objects in $\BB$, for all t sufficiently large, with finite $\mu$-slope which is a contradiction to $\underset{t \rightarrow +\infty}{\lim}{\beta(t)}=-\infty$ and $\underset{t \rightarrow +\infty}{\lim}{\beta(t)}=+\infty$, respectively.  Moreover, if non zero, then $\HH^{-1}(E)$ is a torsion-free sheaf because $E \in \BB$ for some $t$ and this implies that $\HH^{-1}(E)=0$, in either direction of $\gamma$, concluding that $E$ is a sheaf.

The next step is to prove its sheaf stability. Let \begin{equation}\label{Eq-seqnu}
    0 \rightarrow F \rightarrow E \rightarrow G \rightarrow 0
\end{equation} be an exact sequence in $\Coh$ and we can see that $\ch_0(F)=\ch_0(E)=\ch_0(G)=0$, implying that $F,E,G \in \BB$ and that sequence \eqref{Eq-seqnu} is also an exact in $\BB$, for all $t\in \mathbb{R}$. Therefore
\begin{equation}\label{Eq-nu}
\nu_{\gamma(t)}(E)-\nu_{\gamma(t)}(F)= \frac{\delta_{12}(E,F)}{\ch_1(F)\ch_1(E)},
\end{equation}
which is not dependent of $t$, proving the $\gs_1$-(semi)stability of $E$.

Assume now that $E$ is a $2$-dimensional sheaf $\gs_1$-(semi)stable. Being a $2$-dimensional implies that $E \in \BB$ for all $t$, since every quotient of $E$ in $\Coh$ is also a $2$-dimensional sheaf and therefore have infinite $\mu$-slope. Assume now that we have a sequence as \eqref{Eq-seqnu} but in $\BB$ such that $F,E,G \in \BB$ for all $t$ greater than some $t_0$, by applying the argument we started the proof it is clear that both $F$ and $G$ are sheaves and from \eqref{Eq-nu} we conclude asymptotic $\nu_\gamma$-(semi)stability.

\end{proof}

\section{Stability at $-\infty$}

We start the study of asymptotic stability by analyzing the left-hand side of $\mathbb{H}$. It turns out that asymptotic stability is much simpler on this side of the half-plane, because objects in $\Ag$, for $t$ sufficiently large, are coherent sheaves and asymptotic stability is equivalent to $\gs_k$-stability. Throughout this section we will be studying the unbounded curves established in \eqref{Eq-Unbou} but with a new condition: $\underset{t \rightarrow +\infty}{\lim}\beta(t)=-\infty$. 

The first result we prove is related to what kind of object can appear at infinity when considering $\Ag$, for $t$ sufficiently large. Turns out the large volume limit objects are exactly sheaves, and their (semi)stable objects the Gieseker-(semi)stable ones.

\begin{prop}\label{Pro-la1}
An object $E \in \Db$ is in $\Ag$ for every $t$ sufficiently large  if and only if $E \in \Coh$.
\end{prop}

\begin{proof}
Let us start assuming that $E \in \Ag$ for $t$ sufficiently large. We know that $\HH^{-j}(E)=0$ for all $j \neq 0,1,2$ and suppose that $\HH^{-2}(E) \neq 0$, by Lemma \ref{L-ref} we know that $\ch_0(\HH^{-2}(E))>0$ but that would be impossible because $\HH^{-2}(E) \in \FF_{\beta(t)}$, whenever $E \in \Ag$, an absurd because it implies that 
\[0 \geq \underset{t \rightarrow +\infty}{\lim}{\ch_1^{\beta(t)}(\HH^{-2}(E))}=\underset{t \rightarrow +\infty}{\lim}{\ch_1(\HH^{-2}(E))-\beta(t)\cdot\ch_0(\HH^{-2}(E))}=+\infty.\]

Therefore $\HH^{-2}(E)=0$. Moreover, consider $\HH^{-1}(E)$ decomposed by the following exact sequence in $\Coh$
\begin{equation}
    0 \rightarrow \HH^0(\HH^{-1}_{\beta(t)}(E)) \rightarrow \HH^{-1}(E) \rightarrow \HH^{-1}(\HH^0_{\beta(t)}(E)) \rightarrow 0,
\end{equation}
which can change for each value of $t$. Again, applying Grothendieck's Theorem we have that exists $C \in \mathbb{R}$ such that $\ch_1^C(Q)>0$ for all quotients $Q$ of $\HH^{-1}(E)$. If $\HH^{-1}(\HH^0_{\beta(t)}(E))\neq 0$, because it is in $\FF_{\beta(t)}$, it needed to satisfy $\ch_1^{\beta(t)}(\HH^{-1}(\HH^0_{\beta(t)}(E))) \leq 0$ for sufficiently large $t$, which is impossible. Concluding that $\HH^{-1}(\HH^0_{\beta(t)}(E))=0$.

It is clear now that $\HH^{-j}_{\beta(t)}(E)$ are, eventually, constant with respect to $t$. Fixing $t_0$ as the value for which $\HH^{-j}_{\beta(t)}(E)$ are constant and $E \in \Ag$ for every $t>t_0$. In this ray, if $\HH^{-1}_{\beta(t)}(E)=\HH^0(\HH^{-1}_{\beta(t)}(E))$ is non zero we would have
\begin{equation}\label{Lim-nu}
\underset{t \rightarrow +\infty}{\lim}\frac{1}{(-\beta(t))}\nu_{\gamma(t)}(\HH^{-1}_{\beta(t)}(E))= \left\{
	\begin{array}{lll}
	    (1-c_\gamma)  & \mbox{if } \ch_0(\HH^{-1}_{\beta(t)}(E)) \neq 0, \\
		1 & \mbox{if } \ch_0(\HH^{-1}_{\beta(t)}(E))=0\text{, } \ch_1 \neq 0, \\
		+\infty & \text{otherwise}.
	\end{array}
    \right.
\end{equation}
and all possible results contradict $\HH^{-1}_{\beta(t)}(E) \in \FF_{\gamma(t)}$ for $t>t_0$, because $c_\gamma<1$. \\

Now we turn to the case $E \in \Coh$. Using the Harder--Narasimhan filtration for coherent sheaves we obtain a filtration
\[
E_0\subset E_1 \subset E_2 \subset ... \subset E_n=E
\]
where $E_0$ is the maximal torsion subsheaf of $E$, $\Tilde{E}_i:=E_i/E_{i-1}$ are Gieseker-semistable sheaves with reduced Hilbert polynomial $p_i$ and they satisfy $p_i>p_{i+1}$, for $i>0$. If $\ch_1(E_0) \neq 0$ we can determine a Harder--Narasimhan filtration for $E_0$ as $(E_0)_i$, satisfying the same conditions with $\widetilde{(E_0)}_i$ as its Gieseker-semistable sheaves, and in this case the dimension for the maximal torsion sheaf $(E_0)_0$ is at most $1$, therefore $(E_0)_0$ is in $\Ag$ for all $t$. Using the equalities in \eqref{Lim-nu} and Proposition \ref{Pro-nu} we can conclude that all $\widetilde{(E_0)}_i$ are in $\Ag$ for $t>t_0$ and $i>0$, for some $t_0$.

Inductively, all $(E_0)_i$ are in $\Ag$ as they are extensions of $(E_0)_{i-1}$ and $\widetilde{(E_0)}_i$. Applying the same argument to $\Tilde{E}_i$ and $E_i$ we conclude that $E \in \Ag$. If $\ch_1(E_0)=0$ then $E_0 \in \Ag$ and we can skip using its Harder--Narasimhan filtration to prove that $E \in \Ag$.

\end{proof}

As in the previous section, we are able to provide a characterization of asymptotic stability but now in the case of $\phi_\gamma=\lambda_\gamma$. It is important to note that the proof only relies on $c_\gamma<1$ when we apply Proposition \ref{Pro-la1}, implying that this condition is necessary for controlling the structure of the objects and not their $\lambda$-slope. 

\begin{mthm}\label{Teo-Left}
Suppose that $E \in \Db$ with $\ch_0(E)=0$, for $E$ to be asymptotic $\lambda_{\gamma,s}$-(semi)stable it is necessary and sufficient that $E \in \Coh$ is a Gieseker-(semi)stable sheaf.
\end{mthm}

\begin{proof}
From Proposition \ref{Pro-la1} it is known that $E \in \Coh$ is equivalent to $E \in \Ag$ for all $t$ bigger than some $t_0$. Therefore, $F \hookrightarrow E \twoheadrightarrow G$ is an exact sequence in $\Ag$ for all $t$ sufficiently large if and only if it is an exact sequence in $\Coh$. So, in either one of the implications of the theorem we know that $E\in\Coh$ and let $t_0$ be such that $E \in \Ag$ for all $t>t_0$.

Suppose that $E$ has a torsion subsheaf $F$, by the previous observation it is clear that $F$ is also a subobject of $E$ in $\Ag$ for all $t>t_1\geq t_0$, for some $t_1$, and if $\ch_2(F)\neq 0$ then
\[
\underset{t \rightarrow +\infty}{\lim}\frac{1}{(-\beta(t))}(\lambda_{\gamma(t),s}(E)-\lambda_{\gamma(t),s}(F))=(-1)\left(\left(s+\frac{1}{6}\right)c_\gamma + \frac{1}{2}\right)<0,
\]
which would contradict asymptotic $\lambda_{\gamma,s}$-(semi)stability. The case where $\ch_2(F)=0$ would also contradict because $\lambda_{\gamma(t),s}(F)=+\infty$ for all $t$.

Now in both implications of the theorem $E \in \Coh$ and is a pure sheaf. We just have to compare their stabilities. This is done by the following inequalities for the case $\ch_1(E)\neq 0$:

If $\delta_{21}(E,F)\neq 0$:
\[
\underset{t \rightarrow +\infty}{\lim}(\lambda_{\gamma(t),s}(E)-\lambda_{\gamma(t),s}(F))=\delta_{12} \frac{\left(s+\frac{1}{6}\right)c_\gamma + \frac{1}{2}}{\ch_1(E)\ch_1(F)}\geq 0.
\]

If $\delta_{21}(E,F)=0$ and $\delta_{31}(E,F)\neq 0$:
\[
\underset{t \rightarrow +\infty}{\lim}{(-\beta(t)) \cdot (\lambda_{\gamma(t),s}(E)-\lambda_{\gamma(t),s}(F))}=\frac{\delta_{31}}{\ch_1(E)\ch_1(F)} \geq 0.
\]

If both $\delta_{21}$ and $\delta_{31}$ are zero then $\lambda_{\gamma(t),s}(E)=\lambda_{\gamma(t),s}(F)$ for all $t$.

For the case where $\ch_1(E)=0$ and $\ch_2(E)\neq 0$ we have:
\[
\underset{t \rightarrow +\infty}{\lim}{\frac{1}{(-\beta(t))}(\lambda_{\gamma(t),s}(E)-\lambda_{\gamma(t),s}(F))}=\frac{\delta_{31}}{\ch_2(E)\ch_2(F)}\geq 0.
\]

The equivalence is proved using the above inequalities and Remark \ref{Gie-delta}.

\end{proof}

\begin{example}
By the discussion after Remark \ref{Gie-delta}, we know that $i_\ast\mathcal{O}_S$ is Gieseker-stable for any $i:S \rightarrow \mathbb{P}^3$ smooth subvariety of $\mathbb{P}^3$. Therefore, using Main Theorem \ref{Teo-Left}, it is clear that $i_\ast\mathcal{O}_S$ is asymptotic $\lambda_\gamma$-stable. If $S=H$ is a hyperplane in $\mathbb{P}^3$ we have the defining distinguished triangle \begin{equation}\label{Triangle}
 \opn \rightarrow i_\ast\mathcal{O}_H \rightarrow \opn(-1)[1] \rightarrow \opn[1]\end{equation}
and since $\mathcal{O}(k)[i]$ are both Tilt and Bridgeland stable by \cite[Proposition 4.1]{Sch1} whenever they are in the correct space according to Example \ref{Ex-walls}, we can determine actual walls for $i_\ast\mathcal{O}_H$:

\begin{figure}[htp]
    \centering
    \includegraphics[width=13cm]{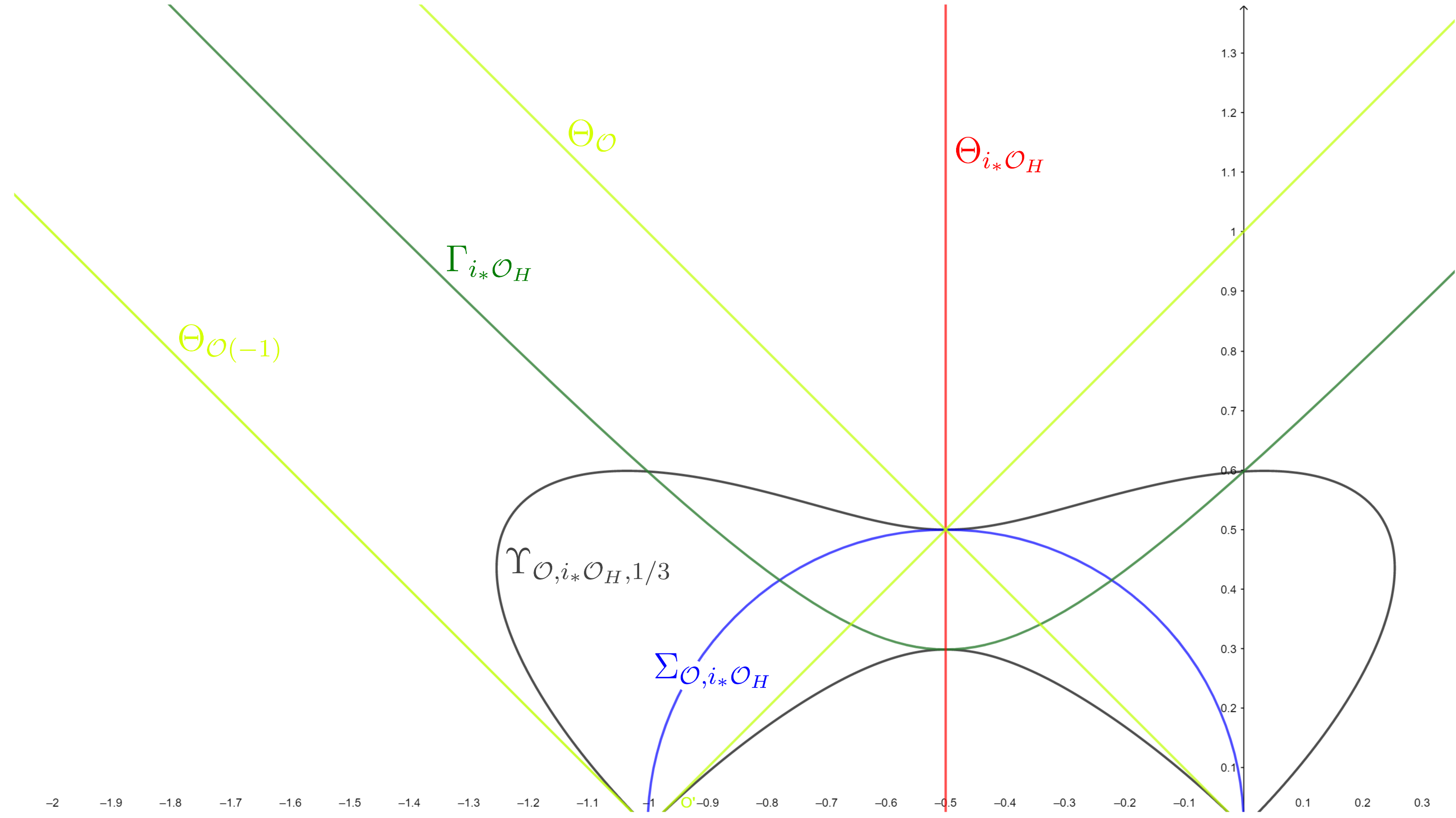}
    \caption{Distinguished curves and walls related to $i_\ast\mathcal{O}_H$ and triangle \eqref{Triangle}, in both tilt and Bridgeland stability, for $s=1/3$.}
\end{figure}

 From this we conclude that $i_\ast\mathcal{O}_H$ is stable right after crossing $\Gamma_{\opn,i_\ast\mathcal{O}_H,s}$ for every $s>0$. This reasoning also works for any hypersurface of degree $d$. We do not know if there are other actual walls destabilizing and stabilizing $i_\ast\mathcal{O}_H$, in spite of knowing that if you go further left enough $i_\ast\mathcal{O}_H$ will be stable because it is asymptotic $\lambda$-stable.
\end{example}

\section{Stability at $+\infty$}\label{Plus-inf}

We turn our attention to the right-hand side of the half-plane $\mathbb{H}$, for that assume the unbounded curve $\gamma$ satisfies: $c_\gamma<1$ and $\underset{t \rightarrow +\infty}{\lim}\beta(t)=+\infty$. 

In this case,  asymptotic $\lambda_\gamma$-stability is more involved than the left-hand side, because the objects $E \in \Ag$, for $t\gg0$, are not sheaves but can be factored by derived duals of sheaves in $\Ag$.  This will be enough to prove a relation between asymptotic $\lambda_\gamma$-stability with Gieseker-stability.

We apply spectral sequences to find conditions for when an object in $\Db$ comes from the derived dual functor $(-)^\vee:=R{\HH}om(-,\opn)[2]$ of a pure sheaf, making the calculations more elaborate. The case where $\ch_0(E)\neq 0$ was already described by Jardim and Maciocia with conditions when this happens in \cite[Main Theorem 3]{JM}.

We will provide these conditions for the case $\ch_0(E)=0$ and $\ch_1(E)\neq 0$. If both $\ch_0(E)=\ch_1(E)=0$, the derived dual applied to a pure sheaf $E$ satisfying these conditions is equal to $E^D:=\mathcal{E}xt^d(E,\opn)$, where $d=\codim(E)$, and as we will see the left-hand and right-hand side of the upper half-plane, in this case, have the same asymptotic $\lambda_\gamma$-(semi)stable objects. 

\begin{lemma}\label{Lem-SpecL}
Let $E \in \Db$ satisfying:
\begin{enumerate}[label=(\alph*)]
    \item $\HH^{i}(E)=0$ if $i \neq 1,0$;
    \item $F=\HH^{-1}(E)$ is a reflexive sheaf of dimension $2$;
    \item $G= \HH^{0}(E)$ is a dimension $0$ sheaf;
    \item The natural map $f:\mathcal{E}xt^1(F,\opn)\rightarrow \mathcal{E}xt^3(G,\opn)$ is an epimorphism.
\end{enumerate}

\noindent if and only if $E^\vee= \ker(f)$ is a pure sheaf of dimension $2$.
\end{lemma}

\begin{proof}
As in \cite[Lemma 2.14]{JM}, assuming $E$ satisfy the conditions described, we can decompose $E$ in the distinguished triangle

\begin{equation}\label{Eq-Esp}
F[1] \rightarrow E \rightarrow G \rightarrow F[2].
\end{equation}

Applying the cohomological functor $R^0\HH om(-,\opn)$ to \eqref{Eq-Esp} we find that $E^\vee=\ker(f)$. To see that $E^\vee=\ker(f)$ imply these properties for $A$ we just have to dualize $\ker(f)$ and see that $\HH^i(E)=\HH^{i}((\ker(f))^\vee)=\mathcal{E}xt^{i+2}(\ker(f),\opn)$ satisfy (a),(b) and (c) because $\ker(f)$ is a pure 2-dimensional sheaf. To see property (d) we apply the spectral sequence
\[
E^{p,q}_2=\mathcal{E}xt^p(\HH^{-q}(E),\opn) \implies \HH^{p+q-2}(\ker(f))
\]
and for this convergence to happen we need the map $f$ to be an epimorphism.
\end{proof}

\begin{rmk}
The case where $\ch_0(E)=\ch_1(E)=0$ is easier to see that $E=A^\vee$, for some pure sheaf $A$ with $\dim(A) \leq 1$, if and only if $\mathcal{H}^i(E)=0$ for $i \neq 0$ and $\mathcal{H}^{0}(E)$ is a pure sheaf with dimension less or equal to $1$. This is because in dimension less or equal to $1$, being a pure sheaf is the same as being reflexive, see \cite[Proposition 1.1.10]{HL}.
\end{rmk}

Next we find conditions every object in $\Db$ has to satisfy in order to be in $\Ag$ for every $t\gg0$. We will need a "right-hand side" version of the equality in display \eqref{Lim-nu}, consider $F \in \Db$ and we would have

\begin{equation}\label{Lim-nu-right}
\underset{t \rightarrow +\infty}{\lim}\frac{1}{\beta(t)}\nu_{\gamma(t)}(F)= \left\{
	\begin{array}{lll}
	    (c_\gamma-1)  & \mbox{if } \ch_0(F) \neq 0, \\
		-1 & \mbox{if } \ch_0=0\text{,} \ch_1(F) \neq 0, \\
		+\infty & \text{otherwise}.
	\end{array}
    \right.
\end{equation}

This equation, for the case $\ch_0\neq 0$, is what justifies the need of the condition $c_\gamma<1$.

\begin{lemma}\label{Lem-CondL}
Suppose $E \in \Db$ is in $\Ag$ for all $t$ sufficiently high. Therefore,

\begin{itemize}
    \item $\HH^{-2}(E)=0$,
    \item $\HH^{-1}(E)=\HH^0(\HH^{-1}_{\beta(t)}(E))$ is either a pure $2$-dimensional sheaf or zero,
    \item $\HH^0(E)=\HH^0_{\beta(t)}(E)$ with $\dim(\HH^0(E))\leq 1$ .
\end{itemize}
\end{lemma}

\begin{proof}
As in the case where $-\infty$, we do not know a priori that $\HH^{i}_{\beta(t)}(E)$ eventually becomes constant. To deal with this technical problem we start by considering that $E \in \Ag $ implies $\Im(Z_{\gamma(t),s}(E))=\ch_2(E)-\beta(t)\ch_1(E)\geq 0$, for $t>t_0$. This is equivalent to $\ch_1(E)\leq0$.

Also, it is known that $\HH^0(E) \in \TT_{\beta(t)}$ for $t>t_0$ and therefore $\ch_0(\HH^0(E))=0$, implying also that $\ch_0(\HH^{-1}(E))=\ch_0(\HH^{-2}(E))$. Suppose  $\HH^{-1}(\HH^{0}_{\beta(t)}(E))\neq 0$ and examine the exact sequence
\begin{equation}
    0 \rightarrow \HH^0(\HH^{-1}_{\beta(t)}(E)) \rightarrow \HH^{-1}(E) \rightarrow \HH^{-1}(\HH^0_{\beta(t)}(E)) \rightarrow 0
\end{equation}
and let $K=\mu(\mathcal{H}^{-1}(E))<\infty$. We can apply Grothendieck's theorem once again to show that
\[
\mathcal{S}:=\left\{p(Q) | \HH^{-1}(E) \twoheadrightarrow Q : {
	\begin{array}{ll}
	    \text{$Q$ is torsion free,}\\
		\text{$\mu(Q)\leq K$}
	\end{array}
    }\right\}
\]
is a finite set, where $p(Q)$ is the Hilbert polynomial of the sheaf $Q$. Moreover,
\[
\Tilde{\mathcal{S}}=\left\{ p(F) | f: F \hookrightarrow \mathcal{H}^{-1}(E) : {
	\begin{array}{ll}
	    \text{$\coker(f)$ is torsion free,}\\
		\text{ $\mu(F)\geq K$}
	\end{array}
    }\right\}.
\]
is also finite 

Now, applying this to the fact $\HH^0(\HH^{-1}_\beta(t)(E)) \in \TT_{\beta(t)}$ for $t>t_0$, while using that $\HH^{-1}(\HH^{0}_{\beta(t)}(E)) \in \FF_{\beta(t)}$, we can conclude that $\ch_0(\HH^0(\HH^{-1}_{\beta(t)}(E)))=0$ for $t$ sufficiently large. Actually, we can also conclude that $\HH^0(\HH^{-1}_{\beta(t)}(E))$ is the maximal torsion subsheaf in $\HH^{-1}(E)$ because its quotient is torsion free and the maximal torsion subsheaf is unique. This uniqueness implies that $\ch_k(\HH^i(\HH^{-j}_{\beta(t)}(E)))$ is fixed for all $i,j,k$ and $t$ sufficiently large. 

With this we can apply the equality in display \eqref{Lim-nu-right} to $\mathcal{H}^0_{\beta(t)}(E)$ in order to prove that \linebreak $\ch_0(\HH^{-1}(\HH^{0}_{\beta(t)}(E)))=0$. Since both $\HH^{-1}(\HH^{0}_{\beta(t)}(E))$ and $\HH^{-2}(E)$ are torsion free sheaves with the same $\ch_0$ character we can conclude that both are zero. 

We only have to prove that $\HH^0(E)$ is not $2$-dimensional and that $\HH^{-1}(E)$ is either pure or zero. The first assertion comes from equality \eqref{Lim-nu-right} applied to $\HH^0(E) \in \TT_{\gamma(t)}$. For the second one assume that $\HH^{-1}(E)$ is a non-pure $2$-dimensional sheaf then exist a subsheaf $T \hookrightarrow \HH^{-1}(E)$ with $\dim(T) \leq 1$, but this is also a subobject of $\HH^{-1}(E)$ in $\BB$ which is impossible because $\HH^{-1}(E)=\HH^{-1}_{\beta(t)}(E) \in \FF_{\gamma(t)}$ for $t\gg0$ and $\nu_{\beta,\alpha}(T)=+\infty$. From this argument, if $\HH^{-1}(E)$ had dimension less than $2$ we would conclude that $\HH^{-1}(E)=0$.
\end{proof}

The last lemma is a reduction on the kind of object we need to test for asymptotic stability.

\begin{lemma}\label{Lem-Deco}
Suppose $Q$ is an object in $\Ag$ for $t\gg0$ with $\ch_0(Q)=0$, $\ch_1(Q)\neq 0$ and $\dim(\HH^0(Q))=0$. Then there is a $2$-dimensional pure sheaf $K$ and a $0$-dimensional sheaf $L$ satisfying the exact sequence
\[
0 \rightarrow L \rightarrow Q \rightarrow K^\vee \rightarrow 0
\]
in $\Ag$ for $t\gg0$.
\end{lemma}

\begin{proof}
By Lemma \ref{Lem-CondL} we know that $\HH^{-2}(Q)=0$, $Q_1:=\HH^{-1}(Q)$ is pure $2$-dimensional sheaf and $Q_0=\HH^0(Q)$ a $0$-dimensional sheaf. Suppose that $Q_1$ is not reflexive and we can use the exact sequence
\[
0 \rightarrow Q_1 \rightarrow Q_1^{DD} \rightarrow L \rightarrow 0,
\]
where $L$ is the $0$-dimensional singularity sheaf associated to $Q_1$. This is an exact sequence in $\Coh$ and $\BB$ for all $t$,  we just have to consider the $\nu_{\gamma(t)}$-slope of subobjects $F$ of $Q_1^{DD}$ in $\BB$. This is done via the following diagram

\centerline{
\xymatrix{
    F' \ar@{^{(}->}[r] \ar@{^{(}->}[d] &F \ar@{->>}[r] \ar@{^{(}->}[d] &I \ar@{^{(}->}[d] \\
    Q_1 \ar@{^{(}->}[r] &Q_1^{DD} \ar@{->>}[r]  &L
}}
\noindent proving that $Q_1^{DD} \in \FF_{\gamma(t)}$ whenever $Q_1 \in \FF_{\gamma(t)}$, because $I \hookrightarrow L$ is a $0$-dimensional sheaf and $Z_{\beta,\alpha}^t(I)=0$ for all $(\beta,\alpha) \in \mathbb{H}$ making $Z_{\gamma(t)}^t(F)=Z_{\gamma(t)}^t(F')$. Leaving us with the exact sequence
\[
0 \rightarrow L \rightarrow Q_1[1] \rightarrow Q_1^{DD}[1] \rightarrow 0 
\]
in $\Ag$. Let $W$ be the cokernel of the composition $L \hookrightarrow Q_1[1]$ and $Q_1[1] \hookrightarrow Q$ in $\Ag$. By the construction of $L$ we know that $W$ satisfy conditions (a),(b),(c) in Lemma \ref{Lem-SpecL}, if $W$ were to satisfy condition (d) we would conclude the proof. Fix $W_1=\HH^{-1}(W)$ and $W_0=\HH^0(W)$.

Assume instead that $f: \mathcal{E}xt^1(W_1,\opn) \rightarrow \mathcal{E}xt^3(W_0,\opn)$ is not surjective and let $\Tilde{P}=\coker(f)$, $\Tilde{L}=\ker(f)$, $\tilde{I}=\im(f)$ such that they are defined by the exact sequence 

\hspace{25mm}{\xymatrix{
0 \ar[r] &\tilde{L} \ar[r] &W_1^D \ar[rr]^f \ar[rd] &&W_0^D \ar[r] &\tilde{P} \ar[r] &0\\
&&&\tilde{I} \ar[ur]
}}

Now, in $\Db$, we can dualize the exact sequence defining $\Tilde{P}$ as a cokernel in $\Coh$, keeping in mind that every $0$-dimensional sheaf is reflexive, and compose with the distinguished triangle defining $W$ in $\Ag$ to obtain

\hspace{30mm}
\xymatrix{
&&\Tilde{P}^D \ar[d] \ar[rd]^{h} \\
W_1[1] \ar[r] &W \ar[r] &W_0 \ar[r] \ar[d] &W_1[2]\\
&&\tilde{I}^D.
}

\noindent After applying the dualization functor to the right part of the diagram we see that we have the commutative diagramm:

\hspace{65mm}
\xymatrix{
\tilde{P}\\
W_0^D \ar[u] &W_1^D \ar[l]^f \ar[ul]_{h^\vee}
}

\noindent implying that $h^\vee$ is the composition of $f$ with its cokernel, making $h^\vee=0$. As a consequence, the map $\tilde{P}^D \rightarrow W_0$ lifts to $g: \tilde{P}^\vee \rightarrow W$

Moreover, due to $\tilde{P}$ being a $0$-dimensional sheaf, it is clear that the cone of $g$ in $\Db$, $C(g)$, is also in $\Ag$ whenever $W \in \Ag$, making $\Tilde{P}$ a subobject of $W$ in $\Ag$ such that $C(g)$ satisfy all conditions in Lemma \ref{Lem-SpecL} making it a dual of a pure $2$-dimensional sheaf $K$. Since the kernels of both maps $Q \rightarrow W$ and $W \rightarrow C(g)$ are $0$-dimensional sheaves, it is clear that $L:=\ker(Q \rightarrow C(g))$ is a $0$-dimensional concluding the proof.
\end{proof}

\begin{rmk}\label{Rmk-reduc}
One application of the previous Lemma is that we can verify asymptotic $\lambda_{\gamma(t)}$-(semi)stability only by considering quotients which are duals of $2$-dimensional pure sheaves because
\begin{equation}
    \underset{t \rightarrow +\infty}{\lim} \beta^l( \lambda_{\gamma(t),s}(Q)-\lambda_{\gamma(t),s}(K^\vee))=\left\{
	\begin{array}{ll}
	    \frac{\ch_3(L)}{\ch_1(K^D)}\geq 0  & \mbox{if } l=1, \\
		0 & \mbox{if } l=0, \\
	\end{array}
    \right.
\end{equation}
such that $\lambda_{\gamma(t),s}(E) \leq \lambda_{\gamma(t),s}(Q)$ if and only if $\lambda_{\gamma(t),s}(E) \leq \lambda_{\gamma(t),s}(K^\vee)$ for $t$ sufficiently large. This is the Bridgeland stability equivalent to \cite[Proposition 1.2.6]{HL}, where it is shown that we can test Gieseker-(semi)stability only using pure quotients.
\end{rmk}

\begin{mthm}\label{Main-2}
An object $E \in \Db$ with $\ch_0(E)=0$ is asymptotic $\lambda_{\gamma}$-(semi)stable if and only if it is the dual of a Gieseker-(semi)stable sheaf. 
\end{mthm}

\begin{proof}

\textit{Case $\ch_1(E)\neq 0$:} We start by assuming that $E$ is asymptotic $\lambda_\gamma$-(semi)stable and we already have information on the cohomology of E given by Lemma \ref{Lem-CondL}. Suppose that $\HH^{0}(E)$ is a $1$-dimensional sheaf and observe that
\[
\underset{t \rightarrow +\infty}{\lim}\frac{1}{\beta(t)}\lambda_{\gamma(t),s}(\HH^0(E))=-1.
\]

Not only $\HH^0(E)$, in this case, is a quotient of $E$ but also
\[
\underset{t \rightarrow +\infty}{\lim}\frac{1}{\beta(t)}(\lambda_{\gamma(t),s}(E)-\lambda_{\gamma(t),s}(\HH^0(E)))=c_\gamma\left(s+\frac{1}{6}\right)+\frac{1}{2}>0,
\]
making this a contradiction to $E$'s asymptotic stability. Now we can apply Lemma \ref{Lem-Deco} to find a $0$-dimensional sheaf as a subobject of $E$ and suppose $E$ is not a dual of a sheaf, this subobject would contradict the asymptotic $\lambda_\gamma$-(semi)stability of $E$, making $K=E^\vee$ a pure sheaf of dimension $2$. Let $Q$ be a pure $2$-dimensional quotient of $K$ with kernel $F\in \Coh$. Dualizing the exact sequence in $\Coh$ determined by $K$, $Q$ and $F$ we obtain the distinguished triangle

\begin{equation}\label{Eq-TriDual}
    Q^\vee \rightarrow E \rightarrow F^\vee \rightarrow G^\vee[1]
\end{equation}

in $\Db$. Now, $Q^\vee$ is in $\Ag$ whenever $E$ is in $\Ag$ because \linebreak $\HH^{-1}(Q^\vee) \hookrightarrow \HH^{-1}(E) \in \FF_{\gamma(t)}$. To see that $F^\vee \in \Ag$ we just apply the cohomology functor to \eqref{Eq-TriDual} and study the subobjects $V$ of $F^D$ in $\BB$ using the diagram
\centerline{
\xymatrix{
    U \ar@{^{(}->}[r] \ar@{^{(}->}[d] &V \ar@{->>}[r] \ar@{^{(}->}[d] &I \ar@{^{(}->}[d] \\
    K^D \ar[r] &F^D \ar[r]  &\mathcal{E}xt^2(Q,\opn)
}}
with $\dim(\mathcal{E}xt^2(Q,\opn))=0$ to conclude that $F^D\in \FF_{\gamma(t)}$ whenever $K^D \in \FF_{\gamma(t)}$. To finish, we just have to look at the limits 
\begin{equation}\label{Eq-infLam1}
\underset{t \rightarrow +\infty}{\lim}(\lambda_{\gamma(t),s}(E)-\lambda_{\gamma(t),s}(Q^\vee))=(-1)\delta_{21}\frac{\left(s+\frac{1}{6}\right)c_\gamma + \frac{1}{2}}{\ch_1(E)\ch_1(Q^D)}\leq 0,
\end{equation}
if $\delta_{1,2}(E,F^\vee)\neq 0$ and on the contrary we have
\begin{equation}\label{Eq-infLam2}
\underset{t \rightarrow +\infty}{\lim}\beta(t)\cdot (\lambda_{\gamma(t),s}(E)-\lambda_{\gamma(t),s}(Q^\vee))=(-1)\frac{\delta_{31}}{\ch_1(E)\ch_1(Q^D)}\leq 0,
\end{equation}
concluding that $K=E^\vee$ is a Gieseker (semi)stable object.

Now let $K=E^\vee$ be Gieseker (semi)stable sheaf and remember that $K^D=\HH^{-1}(E)$ is a $\gs_1$-(semi)stable sheaf implying, by Proposition \ref{Pro-nu}, that $K^D$ is asymptotic $\nu_\gamma$-(semi)stable. One consequence of this fact is that $K^D \in \FF_{\gamma(t)}$ for $t\gg0$ because $\underset{t \rightarrow +\infty}{\lim}\nu_{\gamma(t)}(K^D)=-\infty$, therefore $E \in \Ag$ for t sufficiently large. Consider now the quotient $Q=\Tilde{Q}^\vee$ of $E$ in $\Ag$ for $t\gg0$, such that $\Tilde{Q}$ is a $2$-dimensional pure sheaf, and the exact sequence in $\Ag$
\begin{equation}\label{Eq-LamPos}
0 \rightarrow F \rightarrow E \rightarrow \Tilde{Q}^\vee \rightarrow 0.
\end{equation}
Since $F \in \Ag$ whenever $Q$ is a quotient of $E$, we see that $F$ satisfy the properties described in Lemma \ref{Lem-CondL}. Furthermore, we can dualize the sequence \eqref{Eq-LamPos} to obtain that $\HH^{-i}(F^\vee)=0$ whenever $i\neq 1,0$, and by applying the dualizing functor to the sequence decomposing $F$ in $\Ag$ we can conclude that $F^\vee=\Tilde{F}$ is a sheaf. In this case we have
\[
0 \rightarrow \Tilde{Q} \rightarrow K \rightarrow \Tilde{F} \rightarrow 0
\]
and we only have to apply equations \eqref{Eq-infLam1} and \eqref{Eq-infLam2} to finish the proof.

\textit{Case $\ch_1(E)=0$:} Considering $\ch_2(E)\neq 0$ we first assume that $E$ is $\lambda_\gamma$-(semi)stable and use Lemma \ref{Lem-CondL} to see that $E\in \Coh$ and $E$ is pure because a $0$-dimensional sheaf would destabilize $E$. If $F \in \Coh$ is a sheaf of dimension at most $1$ then $F \in \Ag$ for $t\gg0$, consequence of $\mu(F)=+\infty$ and $\ch_1^{\beta}(F)=0$.  Now we just have to consider the equality for a subsheaf $F$ of $E$
\begin{equation}\label{Eq-LamPos0}
    \lambda_{\gamma(t),s}(E)-\lambda_{\gamma(t),s}(F)=\frac{\delta_{32}(E,F)}{\ch_2(F)\ch_2(E)}
\end{equation}
to see that $E$ is Gieseker-(semi)stable. Conversely, if $F \hookrightarrow E$ in $\Ag$ for all $t$ sufficiently large then $\ch_0(F)=\ch_1(F)=0$ because otherwise $\Im(\coker(F \hookrightarrow E))$ would be negative for some $t\gg0$. Therefore, by Lemma \ref{Lem-CondL} again, $F \in \Coh$ and the same is true for its quotient in $\Ag$,  and applying the same equality in display \eqref{Eq-LamPos0} to prove that $E$ is $\lambda_\gamma$-(semi)stable.

If $\ch_2(E)=0$, by applying Lemma \ref{Lem-CondL} we conclude that $E$ is both a Gieseker-semistable sheaf and $\lambda_{\beta,\alpha,s}$-semistable for all $(\beta,\alpha) \in \mathbb{H}$ and $s>0$, in either case of the theorem.

\end{proof}

We finish this section with an example to illustrate our Main Theorem \ref{Main-2}.

Let $i: S \hookrightarrow \mathbb{P}^3$ be a smooth subvariety of codimension $c \leq 2$ with structure sheaf $i_\ast \mathcal{O}_S$. As discussed in Remark \ref{She-stru}, we know that $i_\ast(\mathcal{O}_S)$ is Gieseker-stable and by applying our Main Theorem \ref{Main-2} we can conclude that $(i_\ast \mathcal{O}_S)^\vee$ is asymptotic $\lambda_\gamma$-stable. This sheaf is described splicitly in \cite[Corollary 3.40]{Huy} as \[ (i_\ast \mathcal{O}_S)^\vee\simeq i_\ast\omega_S\otimes \omega_{\mathbb{P}^3}^\ast[2-c],\] where $\omega_S$ and $\omega_{\mathbb{P}^3}$ are the dualizing bundles of $S$ and $\mathbb{P}^3$, respectively. The left-hand side of this isomorphism is the $i_\ast$-image of the relative dualizing bundle with respect to $i$.

\begin{example}
If $C$ is a curve over $\mathbb{P}^3$ then $(i_\ast\mathcal{O}_C)^\vee=(i_\ast\mathcal{O}_C(d))^D$ is asymptotic $\lambda_\gamma$-stable. When $C$ is complete intersection between hypersurfaces of degree $f$ and $g$, then we can find a resolution \begin{equation}\label{last-eq}
     0 \rightarrow \mathcal{O} \rightarrow \mathcal{O}(f) \oplus\mathcal{O}(g) \rightarrow \mathcal{O}(f+g) \rightarrow (i_\ast\mathcal{O}_C)^D \rightarrow 0.
\end{equation} Moreover, we know that $i_\ast\mathcal{O}_C$ is a pure sheaf making $(i_\ast\mathcal{O}_C)^\vee=(i_\ast\mathcal{O}_C)^D$ an asymptotic $\lambda_\gamma$-stable sheaf. Using equation \eqref{last-eq} we can determine the walls $\Upsilon_{\mathcal{O},\mathcal{O}(f)\oplus\mathcal{O}(g),s}$ and $\Upsilon_{\mathcal{O}(f+g),(i_\ast\mathcal{O}_C)^D,s}$, for a fixed $s>0$, these can be described by Figure $3$.

\begin{figure}[htp]
    \centering
    \includegraphics[width=13cm]{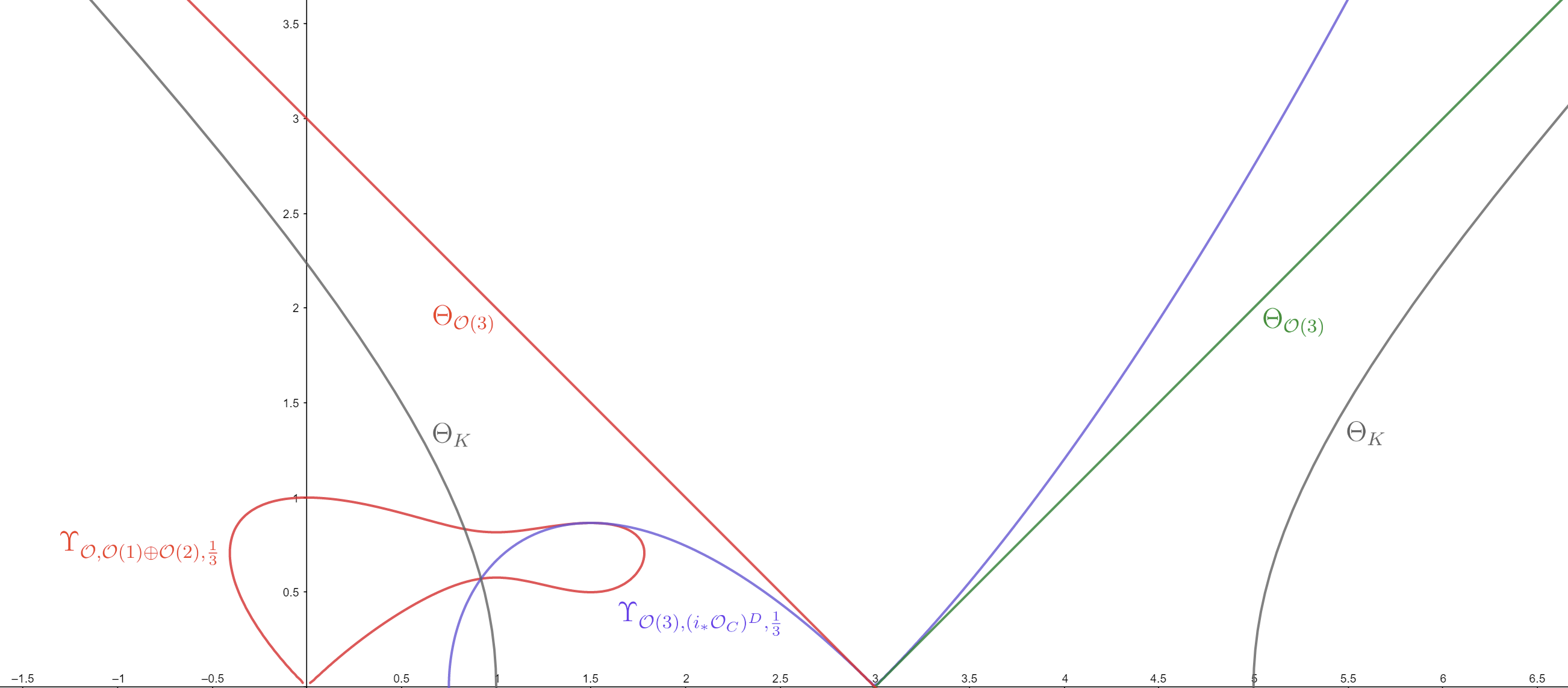}
    \caption{Distinguished curves and walls related to $(i_\ast\mathcal{O}_C)^D$ when $f=1$, $g=2$ and $s=1/3$.}
\end{figure}

Let $(\tilde{\beta},\tilde{\alpha})$ be the point in $\Upsilon_{\mathcal{O},\mathcal{O}(1)\oplus\mathcal{O}(2),\frac{1}{3}} \cap \Upsilon_{\mathcal{O}(3),(i_\ast\mathcal{O}_C)^D,\frac{1}{3}}$ with $\tilde{\beta}=1.5$ and \linebreak $K=\ker(\mathcal{O}(3) \rightarrow (i_\ast\mathcal{O}_C)^D)$ in $\Coh$. We have the exact sequences \[ 0 \rightarrow \mathcal{O}(1)\oplus\mathcal{O}(2)[1]\rightarrow K[1] \rightarrow \mathcal{O}[2] \rightarrow 0 \] \[0 \rightarrow \mathcal{O}(3) \rightarrow (i_\ast\mathcal{O}_C)^D \rightarrow K[1] \rightarrow 0\] in $\mathcal{A}^{\tilde{\beta},\tilde{\alpha}}$ such that $(i_\ast\mathcal{O}_C)^D$ is Bridgeland stable right after crossing the intersection of these walls, towards $\beta\rightarrow+\infty$, but it is still not clear how to conclude that there is no other wall destabilizing $(i_\ast\mathcal{O}_C)^D$ going further from the origin.

\end{example}

\bibliographystyle{abbrv}
\bibliography{bibliografia}

\end{document}